\documentclass{amsart}
\usepackage{amssymb}
\usepackage{graphicx}
\usepackage[all]{xy}
\usepackage{fullpage}
\newtheorem{thm}{Theorem}

\newtheorem{lem}[thm]{Lemma}
\newtheorem{prop}[thm]{Proposition}

\theoremstyle{definition}

\theoremstyle{remark}

\newcommand{\Gal}{\operatorname{Gal}}
\newcommand{\chr}{\operatorname{char}}

\newcommand{\supp}{\operatorname{supp}}
\begin{document}

\title[A proof for a part of noncrossed product theorem]{A proof for a part of noncrossed product theorem}%
\author{Mehran Motiee}%
\address{Faculty of Basic Sciences, Babol Noshirvani University of Technology, Babol, Iran}%
\email{motiee@nit.ac.ir}%


\begin{abstract}
The first examples of noncrossed product division algebras were given by Amitsur in 1972. His method is based on two basic steps:
(1) If the universal division algebra $U(k,n)$ is a $G$-crossed product 
then every division algebra of degree $n$ over $k$ should be a $G$-crossed product; (2) There are 
two division algebras over $k$ whose maximal subfields do not have a common Galois group. In this note, we give a short proof 
for the second step in the case where $\chr k\nmid n$ and $p^3|n$.
\end{abstract}

\maketitle

Let $F$ be a field. An $F$-central division algebra 
$D$ is called a crossed product if it contains a maximal subfield $K$ which is Galois over $F$.
If $\operatorname{Gal}(K/F)\cong G$, then $D$ is said to be a $G$-crossed product. 
Until 1972, the existence of a noncrossed product division algebra remained unsolved and 
influenced by K\"{o}th's theorem, which says that every division algebra contains a maximal subfield 
which is separable over its center, many mathematicians believed that the answer to this question is negative.
Nevertheless, the first examples of noncrossed products 
were given by Amitsur in 1972 \cite{Am72}. His examples of noncrossed
products are certain universal
division algebras which are defined as follows: Let $k$ be an infinite field and
let $X=\{ x_{ij}^{\left( r\right) }| 1\leq i,j\leq n,r\geq 2\}$ be a set of
independent commuting variables. Let $k[X]$ be the integral domain of polynomials
in all $x_{ij}^{\left( r\right) }$ with coefficients in $k$ and let
$k(X)$ the field of fractions $k[X]$. For each $r$, the standard generic
$n$ by $n$ matrices over $k$ are defined by
$\xi ^{\left( r\right) }=\left[ x_{ij}^{\left( r\right) }\right] \in M_{n}\left( F\left[ X\right] \right) \subseteq M_{n}\left( F\left( X\right) \right) $.
The $F$-subalgebra of $M_{n}\left( F\left( X\right) \right)$ generated by
$\{  \xi ^{\left( k\right) }| r < n\} $ is called the generic matrix algebra over $k$ of degree $n$. 
We denote this generic algebra by $B_n$.
It is well-known that $B_n$ is a (noncommutative) domain \cite[Prop. 20.5]{Pi82}, so its center is an
integral domain. We define the universal division algebra over $k$ of degree $n$ by
$UD(k,n)=B_n\otimes_{R}L_n$ where $R$ is the center of
$B_n$ and $L_n$ is the fraction field of $R$. Using the theory of PI-rings, it can be shown that
$UD(k,n)$ is a division algebra of degree $n$ (cf. \cite[Prop. 20.8]{Pi82}). The following theorem is the key result
that Amitsur proved in order to show that certain $UD(k,n)$ are noncrossed products. For a proof 
of this theorem in its full generality\footnote{In his original paper Amitsur only considers the case $k=\mathbb{Q}$, the field of rational numbers.}
see \cite[p. 417]{Pi82}. 
\begin{thm}\label{t1}
If $UD(k,n)$ is a $G$-crossed product, then every division algebra of degree $n$
 whose center contains a subfield isomorphic to $k$ is also a $G$-crossed product.
\end{thm}
In light of Theorem \ref{t1}, to prove that $UD(k,n)$ is not a crossed
product it suffices to produce two different examples of division algebras over $k$ whose
maximal subfields do not have a common Galois group. The aim of this note is to present a simple
proof with least possible computations for this step in the case $\chr k\nmid n$ and $p^3|n$ for some prime $p$.
We note that this case is the most general case independent of the properties of 
$k$ in which it is proved that $UD(k,n)$ is not a crossed product (see \cite{Am74}).
In our approach we employ Skolem-Noether Theorem and a basic property of Henselian fields.
We emphasize that our method neither leads to a new insight into the noncrossed product theorem nor enriches the theory of valued division algebras.
This is just a shortcut for proving a well-known theorem avoiding standard results about the Galois theory of Henselian fields.

First, we recall some preliminaries from the theory of valued division algebras. 
Let $\Gamma$ be a totally ordered additive abelian group. Let $\infty$ is a symbol such that $\gamma<\infty$ and 
$\gamma+\infty=\infty+\infty=\infty$ for all $\gamma\in \Gamma$. By a valuation on $D$ we mean a function 
$$v:D\to \Gamma\cup\{\infty\}$$ satisfying the following conditions: for all $x,y\in D$ 
\begin{enumerate}
  \item [(i)] $v(x)=\infty$ if and only if $x=0$;
  \item [(ii)] $v(x+y)\geq\min\{v(x),v(y)\}$;
  \item [(iii)] $v(xy)=v(x)+v(y)$.
\end{enumerate} 
To a valuation $v$ on $D$ we associate the following structures: $\Gamma_D=v(D^*)$, the value group of $D$, which is a subgroup of $\Gamma$. 
$\mathcal{O}_D=\{x \in D \mid v(x) \geq 0\}$, which is a subring of $D$. $\mathcal{O}_D$ is called the valuation ring of $D$. It is 
easy to observe that $\mathcal{O}_D$ is a local ring with unique maximal left ideal $\mathfrak{m}_D=\{x \in D \mid v(x)>0\}$; so 
$\mathfrak{m}_D$ is a two-side ideal of $\mathcal{O}_D$. The quotient division ring $\overline{D}=\mathcal{O}_D/\mathfrak{m}_D$
is called the residue division ring. For any $a\in D$ we write $\overline{a}$ for the image of $a$ in $\overline{D}$. 
Clearly, restricting $v$ to $F$ is a valuation on $F$. Moreover, it is not hard to show that the inclusion map from 
$F$ to $D$ gives an embedding $\overline{F}\hookrightarrow Z(\overline{D})$. Hence $\overline{D}$ is an $\overline{F}$-algebra. 

Let $F$ be a field equipped with a valuation $v$. The valued field $F$ (or the valuation $v$) is called Henselian if $v$ has a unique extension to 
each field $L$ algebraic over $F$. So, it is clear that every finite extension of a Henselian field is Henselian.
For $f=\sum_{i=0}^n a_i X^i \in \mathcal{O}_F[X]$, 
we write $\overline{f}=\sum_{i=0}^n \bar{a}_i X^i \in \overline{F}[X]$ and $f^{\prime}$ for the formal derivative of $f$.
One can observe that Henselian fields have the following property:
for each $f \in \mathcal{O}_F[X]$, 
if there is a $c \in \overline{F}$ with $\overline{f}(c)=0$ but $\overline{f}^{\prime}(c) \neq 0$, 
then there is a unique $a \in \mathcal{O}_F$ with $f(a)=0$ and $\bar{a}=c$ (see \cite[Th. 4.1.3]{eng05}). 

One of the most important features of a Henselian 
field $F$ that will be used in our proof is that if $\chr \overline{F}\nmid n$, then
 $1+\mathfrak{m}_F$ is $n$-divisible, i.e., $(1+\mathfrak{m}_F)^n=1+\mathfrak{m}_F$. To see this, take any $a\in 1+\mathfrak{m}_F$. 
Let $f=X^n-a\in \mathcal{O}_F[X]$. Its image in $\overline{F}[X]$ is $\overline{f}=X^n-\overline{1}$. 
But $\overline{1}$ is a simple root of $\overline{f}$ since $\chr\overline{F}\nmid n$. 
As $F$ is Henselian, $f$ has a root $b\in \mathcal{O}_F$ with $\overline{b}=\overline{1}$. That is $b^n=a$, as desired.

The following lemma is a known fact. However, for convenience of the reader, we provide a proof for it.
\begin{lem}\label{lem2}
  Let $D$ be a central $F$-division algebra of index $n$. Let $v$ be a valuation on $D$ such that the restriction of $v$ to $F$ is Henselian. 
if $\chr\overline{F}\nmid n$ then in $1+\mathfrak{m}_D$ each element has a unique $n$-th root.
\end{lem}
\begin{proof}
At first, we note that if $\chr D>0$ then $\chr \overline{F}=\chr D$ and if $\chr D=0$ then either $\chr \overline{F}=0$ or $\chr \overline{F}>0$. 
In any case, we have $\chr D\nmid n$.

  Let $a\in 1+\mathfrak{m}_D$ and
Let $K=F(a)$. According to what we mentioned above, the restriction of $v$ to $K$ is Henselian because $K/F$ is 
a finite extension. On the other hand, it is clear that $a\in 1+\mathfrak{m}_K$. So, by the argument in the 
last paragraph, $a$ has a $n$-th root in $K$, hence in $D$. It remains to show that this root is unique. 
Let there are $g,h\in \mathfrak{m}_D$ such that $(1+g)^n=a=(1+h)^n$ 
and $g\neq h$. Then, we have
$$
n\left( g-h\right) =\sum ^{n}_{j=2}\begin{pmatrix} n \\ j \end{pmatrix}\left( h^{j}-g^{j}\right).
$$
Set $t=g-h$. Therefore, we can rewrite the above equation as follows
\begin{equation}\label{55}
nt =\sum ^{n}_{j=2}\begin{pmatrix} n \\ j \end{pmatrix}\left( h^{j}-(t+h)^{j}\right).
\end{equation}
But it is clear that $v(t)>0$.
So $v\left( -t^{2}\right) =v\left( t^{2}\right) =2v\left( t\right) > v\left( t\right) $. Likewise $v(th)>v(t)$ and $v(ht)>v(t)$. Hence $v(h^2-(t+h)^2)=v(-t^2-th-ht)>v(t)$.
Similarly, it can be concluded that $v(h^j-(t+h)^j)>v(t)$ for all $j\geq 3$. So we have
\begin{equation}\label{555}
  v\left(\sum ^{n}_{j=2}\begin{pmatrix} n \\ j \end{pmatrix}\left( h^{j}-(t+h)^{j}\right)\right)>v(t).
\end{equation}
On the other hand, since $\chr D\nmid n$, $n$ is invertible in $\mathcal{O}_D$. Therefore, $v(n)=0$, and we have as a result $v(nt)=v(n)+v(t)=v(t)$. Thus, 
from \eqref{55} and \eqref{555} it follows that
$v(t)>v(t)$, which is a contradiction. 
\end{proof}

In what follows, we recall a certain class of division algebras which are a special case of Mal'cev-Neumann construction (see \cite[\S 14]{Lam91}).  
Let $k$ be an arbitrary field and let $n_1,\dots,n_r$ be integers (not necessarily distinct) with $n_i\geq 2$ for all $i$.
let
$$
n=n_1\dots n_r\quad\text{and}\quad m=\operatorname{lcm}(n_1,\dots,n_r).
$$
Assume that $k$ contains a primitive $m$-th root of unity $\zeta_m$, and let
$\zeta_{n_i}=\zeta_m^{m/n_i}$, for all $i=1,\dots,r$; so $\zeta_{n_i}$ is a primitive
$n_i$-th root of unity. Let $x_1,y_1,\dots,x_r,y_r$ be $2r$ independent indeterminates over $k$. 
Consider the ring of iterated Laurent series $k((x_1))((y_1))\dots((x_{r}))((y_{r}))$ with the multiplication defined by the following rules:
\begin{equation}\label{e3}
   \begin{array}{l}
   x_ia=ax_i,\quad y_ia=ay_i\quad \text{for}\ a\in k \\
   x_ix_j=x_jx_i,\quad y_iy_j=y_jy_i \\
   x_iy_j=y_jx_i\quad \text{for}\quad i\neq j \\
   x_iy_i=\zeta_{n_i} y_ix_i.
 \end{array}
\end{equation}
One can observe that with relations defined in \eqref{e3}, $k((x_1))((y_1))\dots((x_{r}))((y_{r}))$ 
is a division algebra of index $n$ and its center is $F=k((x_1^{n_1}))((y_1^{n_1}))\dots((x_r^{n_r}))((y_r^{n_r}))$ (cf. \cite[p. 100]{J75}). 
We denote this division algebra by $\Delta_{2r}(k; n_1,\dots,n_r)$. Recall that nonzero elements of $\Delta_{2r}(k; n_1,\dots,n_r)$ are formal series
\begin{equation}\label{e1-1}
  f=\sum_{(i_1,\dots,i_n)}a_{(i_1,j_1,\dots,i_{r},j_{r})}x_1^{i_1}y_1^{j_1}\dots x_r^{i_r}y_r^{j_r},\quad a_{(i_1,j_1,\dots,i_{r},j_{r})}\in k
\end{equation}
where $\supp (f)=\{(i_1,j_1,\dots,i_r,j_r)|a_{(i_1,j_1,\dots,i_r,j_r)}\neq 0\}$ is a well-ordered subset of the abelian group 
$\oplus_{i=1}^{2r} \mathbb{Z}$ ordered by the right-to-left lexicographic ordering. 
It can be easily seen that the map 
$$v(f)=\min\supp(f)\quad \text{for}\quad 0\neq f\in \Delta_{2r}(k; n_1,\dots,n_r)$$ 
is a valuation on $\Delta_{2r}(k; n_1,\dots,n_r)$ with value group 
$\Gamma_{\Delta_{2r}(k; n_1,\dots,n_r)}=\oplus_{i=1}^{2r} \mathbb{Z}$.
It is also a known fact that the restriction of $v$ to the center is Henselian with value group 
$\Gamma_F=\oplus_{i=1}^r\left(n_i\mathbb{Z}\oplus n_i\mathbb{Z}\right)$. 

To simplify our notations, for each $\alpha=(i_1, j_1,\dots, i_r,j_r)$, we denote the monomial $x_1^{i_1}y_1^{j_1}\dots x_r^{i_r}y_r^{j_r}$ by $x^{\alpha}$.
Using this notation, we can rewrite \eqref{e1-1} in the form 
$$
f=\sum_{\alpha_1<\alpha_2<\dots}a_{\alpha_i}x^{\alpha_i}
$$
where $\supp(f)=\{\alpha_1,\alpha_2,\dots\}$ and $a_{\alpha_1}\neq 0$. In particular, we have
\begin{equation}\label{e1-2}
  \ker (v)=\{a_0+\sum_{0<\alpha_1<\alpha_2<\dots}a_{\alpha_i}x^{\alpha_i}|a_{\alpha_i}\in k^*\ \text{for all}\ \alpha_i\}=k^*\left(1+\mathfrak{m}_{D}\right)
\end{equation}
where $D=\Delta_{2r}(k; n_1,\dots,n_r)$.

Now, we are ready for the following proposition.
\begin{prop}\label{p1}
Let $D=\Delta_{2r}(k; n_1,\dots,n_r)$. Let $K$ be a maximal subfield of $D$ which is Galois over $F$. Let $H=N_{D^*}(K^*)$, the normalizer of 
$K^*$ in $D^*$. Then there is a diagram
of group homomorphisms
\begin{equation}\label{ee}
\xymatrix{
1\ar[r]&H/F^*\left(H\cap(1+\mathfrak{m}_{D})\right)\ar[r]\ar[d]&\oplus_{i=1}^r\left(\mathbb{Z}_{n_i}\oplus \mathbb{Z}_{n_i}\right)\\
&\Gal(K/F)\ar[d]&\\
&1&
}
\end{equation}
with exact row and column. In particular, both $H/F^*\left(H\cap(1+\mathfrak{m}_{D})\right)$ and $\Gal(K/F)$ are abelian.
\end{prop}
\begin{proof}
The map $\bar{v}: H \rightarrow \Gamma_D / \Gamma_F=\bigoplus_{i=1}^r\left(\mathbb{Z}_{n_i} \oplus \mathbb{Z}_{n_i}\right)$ induced by the valuation $v$ has 
kernel $H \cap F^* \ker (v)$. But $F^*\ker (v)=F^*\left(1+\mathfrak{m}_D\right)$ by \eqref{e1-2}. Hence, $\operatorname{ker}(\bar{v})=H \cap F^*\left(1+\mathfrak{m}_D\right)=F^*\left(H \cap\left(1+\mathfrak{m}_D\right)\right)$. This shows that the row of \eqref{ee} is exact. 

For the column of \eqref{ee}, first let $\theta: H \rightarrow \operatorname{Gal}(K / F)$ be the map $h\mapsto \theta_h$ where $\theta_h:a\mapsto h^{-1}ah$ for all $a\in K$. By the Skolem-Noether Theorem \cite[p. 39]{Dra83}, $\theta$ is surjective. Also, 
$\operatorname{ker}(\theta)=C_H(K^*)$, the centralizer of $K^*$ in $H$. Since $K$ is a maximal subfield of $D$, $C_H(K^*)=K^*$. Hence, $\theta$ induces an isomorphism $H / K^* \cong \operatorname{Gal}(K / F)$. Now take any $h \in H \cap\left(1+\mathfrak{m}_D\right)$. Then, as $\left|H / K^*\right|=|\operatorname{Gal}(K / F)|=n$ 
(recall that $n=n_1\dots n_r$ is the index of $D$), one has $h^n \in K^* \cap\left(1+\mathfrak{m}_D\right)=1+\mathfrak{m}_K$. Because $K$ is Henselian, $h^n$ has a $n$-th root in $1+\mathfrak{m}_K$ which is $h$ itself, as this root is unique in $1+\mathfrak{m}_D$ by Lemma \ref{lem2}. So $h\in K$. Therefore, $\operatorname{ker}(\bar{v})=F^*\left(H \cap\left(1+\mathfrak{m}_D\right)\right) \subseteq K^*=\operatorname{ker}(\theta)$. Hence, $\theta$ induces a well-defined map $H / \operatorname{ker}(\bar{v}) \rightarrow \operatorname{Gal}(K / F)$, and this map is surjective as $\theta$ is surjective. Thus, the column of \eqref{ee} is well-defined and exact.
\end{proof}

\begin{thm}
If $\chr k\nmid n$ and $p^3|n$ for some prime $p$, then $UD(k,n)$ is noncrossed product.
\end{thm}
\begin{proof}
On the contrary, suppose that $UD(k,n)$ is a $G$-crossed product for a group $G$ of order $n$. Let $n=p_1\dots p_r$ where
$p_i$ is prime for $i=1,\dots, r$ (note that these prime numbers are not necessarily distinct).
Let $D_1=\Delta_{2r}(k(\zeta_n); p_1,\dots, p_r)$ and $D_2=\Delta_{2}(k(\zeta_n); n)$.
By Theorem \ref{t1}, $D_1$ is a $G$-crossed product. Hence, by Proposition \ref{p1}, $G$ is abelian and each element of $G$ is of prime order. 
Therefore $G\cong \oplus_{j=1}^r\mathbb{Z}_{p_j}$ because $|G|=n$. But $p^3|n$.  So we must have $\operatorname{rank}(G)\geq 3$.
Using Theorem \ref{t1} again shows that $D_2$ is a $G$-crossed product.
Now, the exact diagram \eqref{ee} gives 
$$\operatorname{rank}(G)\leq \operatorname{rank}\left(H/F^*\left(H\cap(1+\mathfrak{m}_{D})\right)\right)\leq\operatorname{rank}\left(\mathbb{Z}_n\oplus\mathbb{Z}_n\right)=2.$$
This contradiction shows that $UD(k,n)$ is not a crossed product.
\end{proof}
It should be noted that what Amitsure proved about the maximal subfields of $\Delta_{2r}(k; n_1,\dots,n_r)$ is more than what 
was done in the above. In fact, he showed that if $k$ is algebraically closed and $n_i$ are prime for all $1\leq i\leq r$, then 
all maximal subfields of $\Delta_{2r}(k; n_1,\dots,n_r)$ are Galois extensions of its center with Galois group isomorphic to 
$\oplus_{i=1}^r\mathbb{Z}_{n_i}$. He obtained the Galois group information by some
clever but mysterious calculations using certain integer-valued triangular matrices.
His method, along with the Platonov's reduced $K$-theory, became a source of inspiration for others to develop 
the theory of valued division algebras. In \cite{Mot}, in an attempt to provide an elementary proof for the above results, 
the author presents a proof without using Henselian property of the fields of iterated Laurent series for the case $n=p^m$ where $p$ is a prime and $m\geq 3$. 
\bibliographystyle{amsplain}

\end{document}